\documentclass[11pt]{amsart}
\usepackage{amssymb,amsfonts,amsmath}
\usepackage[all]{xy}
\usepackage[shortlabels]{enumitem}

\newtheorem{theorem}{Theorem}[section]
\newtheorem{proposition}{Proposition}[section]

\newtheorem{corollary}{Corollary}[section]

\theoremstyle{definition}
\newtheorem{definition}{Definition}[section]

\newtheorem{lemma}{Lemma}[section]
\newtheorem{example}{Example}[section]

\theoremstyle{remark}
\newtheorem{remark}{Remark}

\theoremstyle{claim}
\newtheorem{claim}{Claim}[section]

\theoremstyle{theoremA}

\theoremstyle{theoremB}

\newcommand{\real}{\mathbb{R}}

\newcommand{\Sp}{\mathbf{S}}

\newcommand{\si}{\sigma}

\newcommand{\la}{\lambda}
\newcommand{\al}{\alpha}

\newcommand{\om}{\omega}
\newcommand{\Om}{\Omega}
\newcommand{\na}{\nabla}
\newcommand{\ep}{\epsilon}
\newcommand{\ga}{\gamma}

\newcommand{\be}{\beta}

\newcommand{\de}{\delta}
\newcommand{\ka}{\kappa}
\newcommand{\ti}{\tilde}
\newcommand{\vp}{\varphi}
\newcommand{\vol}{\text{\rm vol}}

\newcommand{\lan}{\left\langle}
\newcommand{\ran}{\right\rangle}
\newcommand{\p}{\partial}

\newcommand{\dv}{\mathrm{div}}
\newcommand{\hs}{\mathrm{Hess}}
\newcommand{\tr}{\mathrm{Tr\,}}

\newcommand{\rad}{\mathrm{rad}}
\newcommand{\m}{\mathcal}

\newcommand{\spp}{\mathrm{supp}\,}
\newcommand{\inj}{\mathrm{Inj}}
\newcommand{\La}{\Lambda}

\title[Sobolev inequality for weighted submanifolds]{Sobolev and isoperimetric inequalities for submanifolds in weighted ambient spaces}
\author{M. Batista}
\address{Instituto de Matem\'atica, Universidade Fe\-deral de Alagoas, Macei\'o, AL, CEP 57072-970, Brazil}\email{mhbs28@impa.br}
\author{H. Mirandola}
\address{Instituto de Matem\'atica, Universidade Fe\-deral do Rio de Janeiro, Rio de Janeiro, RJ, CEP 21945-970, Brasil} \email{mirandola@ufrj.br}

\begin{document}
\maketitle
\begin{abstract} In this paper, we prove Sobolev and isoperimetric inequalities for submanifold in weighted manifold. Our results generalize the Hoffman-Spruck's inequalities \cite{hs}.
\end{abstract}

\section{Introduction}

A lot of topics in the geometric analysis, such as, Ricci flow, mean curvature flow, anisotropic mean curvature and optimal transportation theory, are related to submanifolds in weighted manifolds, see for instance \cite{e}, \cite{cmz}, \cite{mw1}, \cite{mw2}, \cite{ww}, \cite{m} and references therein. We recall that a weighted manifold $(\bar{M},g,d\bar\mu)$ is a Riemannian manifold $(\bar M,g)$ endowed with a weighted volume form  $ d\bar\mu=e^{-f}d\bar M$, where $d\bar M$ is the volume element induced by the metric $g$ and $f$ is a real-valued smooth function on $\bar{M}$, sometimes called the density of $\bar M$. In this paper, following the  papers of Hoffman and Spruck \cite{hs} and Michael and Simon \cite{ms}, we will study  Sobolev and isoperimetric inequalities to immersed submanifolds in weighted ambient spaces. The value of such inequalities is well known in the theory of the partial differential equations.

 Let $x:M\to \bar M$ be an isometric immersion of a complete manifold with (possibly nonempty) boundary $\p M$ in the weighted manifold $(\bar M,g,d\bar\mu)$. Following Gromov \cite{g}, some authors have introduced the extrinsic object associate to the immersion $x$, called by weighted mean curvature vector field $H_f$, given by $$H_f=H+\bar{\nabla}f^\perp,$$ where $H$ is the mean curvature vector of the submanifold and ${}^\perp$ denote the orthogonal projection onto the normal bundle $TM^\perp$. In this context, it is natural to consider the first and second variations for the weighted area functional,
$$\vol_f(\Omega)=\int_\Omega d\mu,$$ where $d\mu= e^{-f(x)}dM$ and $\Om$ is a bounded domain. In 2003, Bayle \cite{b}, obtain the first variational formulae
$$\dfrac{d}{dt}\bigm|_{t=0}\vol_f(\Omega_t)=\int_\Omega \langle H_f, V\rangle d\mu,$$ where $V$ is variational field. Thus the $f$-mean curvature vector appears naturally from a variational context.

\begin{example}\label{self-shrinker example} Consider the weighted Euclidean space $(\real^n, d\bar\mu=e^{-|x|^2/4}dx)$, where $|\!\cdot\!|$ denotes the Euclidean norm and $dx$ the Euclidean volume element. We recall that an isometric immersion $F:M\to \real^n$ is be a self-shrinker if its mean curvature vector satisfies $2H=-F^\perp$. It is simple to show this definition is equivalent to say that $F$ is $({|x|^2}/{4})$-minimal. 
\end{example}

To state our main theorem, we need some definitions and notations. Let $\m K:\real\to [0,\infty)$ be a non-negative even continuous and $h$ the solution of the following Cauchy Problem:
\begin{equation}\label{cauchy} \left\{
\begin{array}{l}
h''+ \m K h = 0\\
h(0)=0, h'(0)=1.
\end{array}\right.
\end{equation}
Let $r_0=r_0(\m K)>0$ and $s_0=s_0(\m K)>0$ be defined as follows: $(0,r_0)$ is an interval where $h$ is increasing and $(0,s_0)=h(0,r_0)$. 
Assume that the radial curvatures of $\bar M$ with base point
$\xi$ satisfy 
\begin{equation}\label{rad-cond}
(\bar K_{\rad})_\xi\leq \m K(r_\xi),
\end{equation}
for all $\xi\in M$, where $r_\xi=d_{\bar M}(\cdot\,,\xi)$ is the distance in $\bar M$ from $\xi$. Our main theorem says the following. 

\begin{theorem}\label{desig-Sobolev} Under the notations above, we assume that $\bar M$ satisfies (\ref{rad-cond})  and that $f^*=\sup_M f<+\infty$. Let $\vp$ be a compactly supported nonnegative $C^1_0$ function on $M$ that vanishes along the boundary $\p M$. Then there exists a positive constant $S$, depending only $m$ and $\m K$ 
such that the following inequality holds:
\begin{equation*}
\left(\int_M \vp^{\frac{mp}{m-p}}d\mu\right)^{\frac{m-p}{m}}
 \leq S\, e^{\frac{f^*}{m}} \int_{M} \left(|\na \vp| + \vp|H_f-\bar\na f|\right)^p d\mu,
\end{equation*}
for all $1\le p<m$, provided that there exists $\ka\in (0,1)$ satisfying: 
\begin{equation}\label{cond-sobolev-true}\left\{
\begin{array}{l}
\bar J:= \left(\dfrac{\om_m^{-1}e^{f^*}}{1-\ka}\vol_f\big(\spp(\vp)\big)\right)^{\frac{1}{m}}\le s_0; \\ \\
h^{-1}(\bar J)\le 2\, \inj_{\vp},
\end{array}\right.
\end{equation}
where $\om_m$ is the volume of the unit ball in $\real^m$ and $\inj_\vp$ is 
 the minimum of the injectivity radius of $\bar M$ restricted to the points of 
 $\spp\vp$. Furthermore, the constant $S$ is given by
\begin{equation}\label{sobolev-constant}
S=\frac{2^m m}{\ka(m-1)}\frac{r_0}{s_0}\left(\frac{\om_m^{-1}}{1-\ka}\right)^\frac{1}{m}. \end{equation}
\end{theorem}
\begin{remark}
It is simple to see that if $\bar M$ is a Hadamard manifold then $\bar R_\vp=+\infty$ and we can take $\m K=0$, hence any solution $h$ of (\ref{cauchy}) is given by $h(t)=t$ defined on any positive interval $(0,r_0)$. Thus condition (\ref{cond-sobolev-true}) is always satisfied and ${r_0}/{s_0}=1$. In this case, we can choose $S=S_0$ by 
\begin{equation}\label{optimal}
S_0 = \min_{k\in (0,1)} S=\frac{2^m (m+1)^{\frac{m+1}{m}}}{m-1}\om_m^{\frac{-1}{m}}.
\end{equation}  
If $\bar M$ is the sphere $\Sp^n(1/b)\subset \real^{n+1}$ of radius $1/b>0$ then we can take $\m K=b^2$. In this case, $h(t)=b^{-1}\sin(tb)$ defined on the interval $(0,\pi/(2b))$. Hence $r_0/s_0=\pi/2$. Thus we see that  Theorem \ref{desig-Sobolev} improve Hoffman-Spruck's inequality \cite{hs} even when $f\equiv 0$. The question on the optimal constant $S$ in  Theorem \ref{desig-Sobolev} remains open, even for $f\equiv 0$ and $M$ being a minimal surfaces in $\real^3$. To more details about this problem see \cite{c,ch}.
\end{remark}

A consequence of Theorem \ref{desig-Sobolev} is the following isoperimetric inequality.

\begin{theorem}\label{des_isoperimetrica} Under the notations above we assume that $\bar M$ satisfies (\ref{rad-cond}) and that $M$ is compact with possibly nonempty boundary.  Then it holds
\begin{equation}\label{eq-des-isoperimetrica}
\vol_f(M)^{\frac{m-1}{m}} \leq  S e^{\frac{f^*}{m}}\left(\vol_f(\partial M) + \int_M |H_f-\bar\na f|d\mu\right),
\end{equation}
provided that there exists $\ka\in (0,1)$ satisfying: 
\begin{equation}\label{cond-isoper-inq}\left\{
\begin{array}{l}
\bar J= \left(\dfrac{\om_m^{-1}e^{f*}}{1-\ka}\vol_f(M)\right)^{\frac{1}{m}}\le s_0; \\ \\
h^{-1}(\bar J)\le 2 \inj_{M},
\end{array}\right.
\end{equation}
where $f^*=\sup_M f$, $\inj_{ M}$ is the minimum of the injectivity radius of $\bar M$ restricted to the points of $M$, and $S$ is the constant as given in (\ref{sobolev-constant}).
\end{theorem}

By Theorem \ref{des_isoperimetrica}, it is simple to show that if $M^m$ is a closed self-shrinkers contained in a Euclidean ball $B\subset \real^n$ of radius $R$ then it holds that $e^{R^2/4}R\ge 2/S_0$ and $\vol_{(|x|^2/4)}(M)^{1/m}\ge {2e^{-R^2/4}}/{(S_0R)}$, where $S_0$ is the positive constant as in (\ref{optimal}).
Since the round spheres $S^m(\sqrt{2m})\subset \real^{m+1}$ of radius $\sqrt{2m}$ are examples of  $(|x|^2/4)$-minimal hypersurfaces, the term $``|H_f-\bar\na f|"$ that appears in Theorems \ref{desig-Sobolev} and \ref{des_isoperimetrica} cannot be replaced by $``|H_f|$". We can also see that the hypothesis $``f^*<\infty"$ is essential in Theorems \ref{desig-Sobolev} and \ref{des_isoperimetrica}. Consider a weighted Euclidean space $(\real^3,e^{-f}dx)$. If we take the function $f(x)=|x|^2/2$ then the plane $P=\real^2\subset \real^3$ has finite $f$-volume, $H_f=0$ and $\bar\na f=x$, hence $|H_f-\bar\na f|$ has finite $L^2_\mu$-norm. However, if $f\in C^1(\real^3)$ satisfies $f^*<\infty$ and $\sup_P|\na f|<\infty$ then, by Theorem \ref{des_isoperimetrica} and coarea formula, we can show that  that $P$ has infinite $f$-volume. More generally, we have the following

\begin{theorem}\label{cmv-proc} Let $\bar M$ be a complete weighted manifold $(\bar M, d\mu=e^{-f}d\bar M)$ with injectivity radius bounded from below by a positive constant and radial sectional curvatures satisfying (\ref{rad-cond}), for some even function $0\le \m K\in C^0(\real)$. Let $M^m$ be a complete noncompact manifold isometrically immersed in $\bar M$. Assume that $f^*<\infty$ and $|H_f-\bar\na f|\in L^p_\mu(M)$, for some $m\le p\le \infty$. Then each end of $M$ has infinite $f$-volume.
\end{theorem}

Theorem \ref{cmv-proc}, for the case that $\bar M$ has bounded geometry, was proved by: (i) Frensel \cite{fr} and by do Carmo, Wang and Xia \cite{cwx} for the case that the mean curvature vector field is bounded in norm (the case $p=\infty$); (ii) Fu and Xu \cite{fx} for the case that the total mean curvature is finite (the case $p=m$); and Cheung and (iii) Leung \cite{cl} for the case that the mean curvature vector has finite $L^p$-norm for some $p>m$.

We were informed of an independent manuscript of Debora Impera and Michele Rimoldi \cite{ir} which proves a similar version of Theorem \ref{desig-Sobolev} for the case that $M$ is a hypersurface in a weighted manifold $\bar M$ with nonpositive sectional curvature. The authors thank them for useful comments. 

\section{Preliminaries}
We assume the notations  in the introduction. Consider the following

\begin{definition} Let $X:M\to T\bar M$ be a $C^1$ vector field. The {\it f-divergence} of $X$ is defined by: 
\begin{equation*}
\m D_f X =e^f\dv_M (e^{-f}X^T)
\end{equation*}
\end{definition}
By a direct computations, the following holds.
\begin{proposition}\label{phi vector} Let $X:M^m\to T\bar M$ be a $C^1$-vector field and 
$g\in C^1(M)$. Then it holds:
\begin{enumerate}[(A)]
\item\label{f tangent} $\m D_f X = \dv_M X +\lan  H - \na f, X\ran= \dv_M X + \lan H_f - \bar\na f, X\ran$, where $\na f=(\bar\na f)^T$ is the gradient vector field of the restriction $f|_M$;
\item\label{f function} $\m D_f (gX)=g\m D_fX + \lan X, \na g\ran$.
\end{enumerate}
\end{proposition}

Fix a point $\xi\in M$ and consider $r_\xi=d_{\bar M}(\cdot\,,\xi)$ the distance function in $\bar M$ from $\xi$. Assume that the radial curvature of $\bar M$ with basis point $\xi$ satisfies 
\begin{equation}\label{radial}
(\bar K_{\rad})_{\xi} \leq \m K(r_\xi),
\end{equation}
where $\m K:\real\to [0,\infty)$ is a non-negative even continuous function. Let $h:(0,r_0)\to (0,s_0)$ be the increasing function as defined in  (\ref{cauchy}).  

Let $\m B=\m B_{r_0}(\xi)$ be the geodesic ball of $\bar M$ with center $\xi$ and radius $r_0$. Consider the radial vector field 
\begin{equation}\label{radial}
X_\xi=h(r_\xi)\bar\na r_\xi,
\end{equation}
defined on $\m B\cap V$, where $V$ is a normal neighborhood of $\xi$ in $\bar M$ and $\bar\na r_\xi$ is the gradient vector field of $r_\xi$ in $\bar M$. By the hessian comparison theorem (see Theorem 2.3 page 29 of \cite{rsp}), we have that in $\m B$ the following holds
\begin{equation}\label{hess-comparison}
\hs_{r_\xi}(v,v)\ge \frac{h'(r_\xi)}{h(r_\xi)}(1-\lan \bar\na r_\xi,v\ran^2),
\end{equation}
for all vector field $v\in T\bar M$ with $|v|=1$.

\begin{proposition}\label{hessiana} Under the notations above, it holds that
\begin{equation}\label{prop_comparison}
\m D_f X_\xi\geq m h'(r_\xi)+h(r_\xi)\lan H_f - \bar\na f,\bar\na r_\xi\ran.
\end{equation} 
\end{proposition}
\begin{proof}
Using Proposition \ref{phi vector} we have 
\begin{equation}\label{Dfx1}
\m D_f X_\xi = h(r_\xi)\m D_f\bar\na r_\xi + h'(r_\xi)|\na r_\xi|^2.
\end{equation}
Furthermore, using (\ref{hess-comparison}), we obtain
\begin{eqnarray}\label{Dfx2}
\m D_f\bar\na r_\xi &=& \dv_M \bar\na r_\xi + \lan  H_f - \bar\na f, \bar\na r_\xi\ran \nonumber\\ &\geq& \dfrac{h'(r_\xi)}{h(r_\xi)}(m-|\na r_\xi|^2)+\lan H_f - \bar\na f , \bar\na r_\xi\ran.
\end{eqnarray}
Combining (\ref{Dfx1}) and (\ref{Dfx2}),  the result follows.
\end{proof}

Let $M$ be a complete manifold with (possibly nonempty) boundary $\p M$ 
and let $\vp:M\to [0,\infty)$ be a compactly supported nonnegative $C^1$ function such that $\vp|_{\p M}=0$. Let $\la\in C^1(\real)$ 
be a non-negative and non-decreasing function satisfying $\la(t)=0$, 
for $t\leq 0$. We define the following real-variable functions:
\begin{equation*}\label{phi and psi}
\begin{array}{l}
\phi_{\xi}(R)=\phi_{\xi,\vp,\la}(R)=\int_M \la(R-r_\xi(x))\vp d\mu;\\
\\
\psi_{\xi}(R)=\psi_{\xi,\vp,\la}(R)=\int_{M}\la(R-r_\xi(x))(|\na \vp+\vp(H_f-\bar\na f )|d\mu;\\
\\
\bar\phi_{\xi}(R)=\bar\phi_{\xi,\vp}(R)=\int_{M\cap \m B_{R}(\xi)} \vp d\mu;\\
\\
\bar\psi_{\xi}(R)=\bar\psi_{\xi,\vp}(R)=\int_{M\cap \m B_R(\xi)} (|\na \vp+\vp(H_f-\bar\na f)|d\mu.
\end{array}
\end{equation*}
Our first lemma says the following.
\begin{lemma}\label{l4.1} It holds that 
\begin{equation*}
-\frac{d}{dR}\left(h(R)^{-m}\phi_{\xi}(R)\right)\leq h(R)^{-m}\psi_{\xi}(R),
\end{equation*}
for all $0<R<R_0=\min\{\inj_\vp,r_0\}$.
\end{lemma}
\begin{proof}  
We denote by $r=r_\xi$ and let $X=X_\xi$ be defined in $\m B_{R_0}(\xi)$. Using \ref{f function} we obtain that
\begin{eqnarray}\label{la1}
\m D_f(\la(R-r)\vp\,X) &=& \la(R-r)\vp \m D_fX + \lan \na(\la(R-r)\vp),  X\ran \\ \nonumber
&=& \la(R-r)\vp\m D_fX + \la(R-r)\lan\na \vp,X\ran \\ && -\la'(R-r)\vp\lan\na r,X\ran. \nonumber
\end{eqnarray}
Since $\spp\vp$ is compact and $\vp|_{\p M}=0$, using  Item \ref{f tangent} of Proposition \ref{phi vector} and the divergence theorem, we obtain
\begin{equation}\label{la2}
\int_M \m D_f(\la(R-r)\vp\,X)d\mu= 0.
\end{equation}
Thus, by (\ref{la1}) and (\ref{la2}), we obtain 
\begin{eqnarray}\label{quase}
\int_M \la(R-r)\vp \m D_fX d\mu &=& \int_M \la'(R-r)\vp h(r)\lan\na r,\bar\na r\ran d\mu\\ &
& - \int_M \la(R-r)h(r)\lan \na \vp, \bar \na r\ran d\mu \nonumber.
\end{eqnarray}
Using that:
\begin{enumerate}[(a)]
\item\label{la e la'} the functions $\la$ and $\la'$ are nonnegative;
\item the function $h$ is positive and increasing in $(0,r_0)$;
\item\label{zero} $\la(R-r(x))=\la'(R-r(x))=0$ in the subset $\{x\in M \mid r(x)\geq R\}$.
\end{enumerate}
Since $h''=-\m K h\leq 0$ in $(0,r_0)$ we have that $h'$ is non-increasing in $(0,r_0)$. By using \ref{la e la'}, \ref{zero} and Proposition \ref{hessiana}, we obtain that 
\begin{equation*}
\int_M \la(R-r)\vp \m D_fX d\mu \ge mh'(R)\phi(R) + \int_M \la(R-r)\vp h(r) \lan H_f-\bar\na f,\bar\na r\ran.
\end{equation*}
Thus, since $|\na r|\le 1$, using (\ref{quase}), \ref{la e la'} and \ref{zero} we obtain
\begin{eqnarray*}
m h'(R)\phi_\xi(R) &\le& h(R) \int_M  \la'(R-r)\vp d\mu \\&& -\, \int_M \la(R-r) h(r)\lan \na\vp + \vp(H_f-\bar\na f), \bar\na r\ran\\&& +\, h(R)\Big(\frac{d}{dR}\phi_\xi(R) + \psi_\xi(R)\Big)
\end{eqnarray*}
This implies that 
\begin{eqnarray*}
\frac{d}{dR} \left(h(R)^{-m}\phi_\xi(R)\right) &=& h(R)^{-m}\Big(\frac{ d\phi_\xi}{d R}(R)-m\frac{h'(R)}{h(R)}\phi_\xi(R)\Big)\\ &\geq&  h(R)^{-m}\Big(\frac{d\phi_\xi}{d R}(R)-\big(\frac{d\phi_{\xi}}{dR}(R) + \psi_{\xi}(R)\big)\Big)\\ &=& -\, h(R)^{-m} \psi_{\xi}(R).
\end{eqnarray*}
Lemma \ref{l4.1} is proved.

\end{proof}

Take $\ka\in (0,1)$ and let $J=J_{(\ka,\vp,f)}\geq 0$ be the constant defined by 
\begin{equation}\label{J} 
J=\left(\frac{\om_m^{-1}e^{f^*}}{1-\ka}\int_M \vp d\mu\right)^{\frac{1}{m}}.
\end{equation}

Our next lemma is the following result.

\begin{lemma}\label{l4.2} Fix $\xi\in M$ satisfying $\vp(\xi)\ge 1$.
Assume that $0<J <s_0$ and set $\al=\al(\ka,\vp)\in (0, r_0)$ given by $h(\al)= J$. Assume further that $t\al \leq R_0$, for some $t>1$.
Then there exists $R\in (0, \al)$ such that
\begin{equation}\label{eql4.2}
\bar\phi_\xi(tR)\leq \frac{2\al}{\ka} t^{m-1}\,\bar\psi_\xi(R).
\end{equation}
\end{lemma}
\begin{proof} 


By Lemma \ref{l4.1},  
\begin{equation}\label{ti-h}
-\frac{d}{dR}(h(R)^{-m}\phi_\xi(R))\leq h(R)^{-m}\psi_\xi(R),
\end{equation}
for all $0<R<R_0$. 

Note that $0<\al\le R_0=\min\{\inj_{\vp},r_0\}$. Given $\si\in (0,\al)$, integrating the both sides of (\ref{ti-h}) on the interval $(\si,\al)$ we obtain
\begin{equation}\label{des_int}
h(\si)^{-m}\phi_\xi(\si) \leq  h(\al)^{-m}\phi_\xi(\al) + \int_{\si}^{\al} h(\tau)^{-m}\psi_\xi(\tau)d\tau.
\end{equation} 

Take $0<\ep<\si$ and let $\la:\real\to [0,1]$ be a nondecreasing $C^1$ function satisfying: 
\begin{equation}\label{lambda} \left\{
\begin{array}{l}
\la(t)= 1, \mbox{ for all } t\geq \ep;\\
\la(t)= 0, \mbox{ for all } t\leq 0;\\
0\leq \la(t)\leq 1, \mbox{ for all } t.
\end{array}\right.
\end{equation} 
Consider this function $\la$ in the definitions of $\phi_{\xi}=\phi_{\xi,\vp,\la}$ and $\psi_\xi=\psi_{\xi,\vp,\la}$. By (\ref{des_int}) and (\ref{lambda}), we obtain
\begin{eqnarray}\label{si}
\phi_\xi(\si)&=&\int_M \la(\si-r_\xi)\vp d\mu = \int_{M\cap \m B_\si(\xi)} \la(\si-r_\xi)\vp d\mu\\ &\geq& \int_{M\cap \m B_{\si-\ep}(\xi)} \la(\si-r_\xi)\vp d\mu=\int_{M\cap \m B_{\si-\ep}(\xi)}\vp d\mu\nonumber \\ \nonumber\\&=&m\bar\phi_\xi(\si-\ep)\nonumber
\end{eqnarray}
Since $0\leq\la(t)\leq 1$, for all $t$, and $\la(R-r_\xi(x))=0$ in $\{x\in M \mid r_\xi(x)\geq R\}$, we have that $\phi_\xi(\si)\leq \bar\phi_\xi(\si)$ and $\psi_\xi(\si)\leq \bar\psi_\xi(\si)$. Thus, by (\ref{des_int}) and (\ref{si}), we obtain the following.
\begin{equation}\label{ep}
h(\si)^{-m}\bar\phi_\xi(\si-\ep)\leq h(\al)^{-m}\bar\phi_\xi(\al) + \int_0^{\al} h(\tau)^{-m}\bar\psi_\xi(\tau)d\tau
\end{equation}
Since the inequality (\ref{ep}) does not depend on $\la$ we can take $\ep\to 0$. Thus we obtain  
\begin{equation}\label{sup}
\sup_{\si\in(0,\al)}\left(h(\si)^{-m}\bar\phi_\xi(\si)\right)\leq h(\al)^{-m}\bar\phi_\xi(\al) + \int_0^{\al} h(\tau)^{-m}\bar\psi_\xi(\tau)d\tau
\end{equation}

Now suppose that Lemma \ref{l4.2} is false. Then it holds that 
\begin{equation*}\label{false}
\bar\psi_\xi(R) <\frac{\ka}{2\al}t^{1-m}\bar\phi_\xi(tR),
\end{equation*}
for all $R\in (0,\al)$. Multiplying the both sides of this inequality by $h(R)^{-m}$, integrating on $(0,\al)$ and using the change of variable $\si=tR$ we obtain  
\begin{equation}\label{change}
\int_0^{\al} h(R)^{-m}\bar\psi_\xi(R)dR <\frac{\ka}{2\al}t^{-m}\int_0^{t\al}\big(h\big(\frac{\si}{t}\big)\big)^{-m}\bar\phi_\xi(\si)d\si.
\end{equation}
Given $0<\si<t\al \leq R_0$, using that  $h''=-\m K h\leq 0$ we have that $h$ is concave and increasing on $(0, \al)$. Thus we obtain the following.
\begin{equation}\label{if}\left\{
\begin{array}{l}
\mbox{If } \si\in (0, \al) \mbox{ then } h(t^{-1}\si)\geq t^{-1} h(\si), \mbox{  for all } t\geq 1;\\ \\
\mbox{If } \si\in (\al,t\al) \mbox{ then } 0<\frac{\si}{t\al}<1 \mbox{ and } \frac{\si}{t}=\frac{\si}{t\al}\al,\\ \mbox{ which implies that } h(\frac{\si}{t})\geq \frac{\si}{t\al}h(\al).
\end{array}\right.
\end{equation}
Using (\ref{if}) we obtain
\begin{eqnarray*}
\int_0^{t\al} \big(h\big(\frac{\si}{t}\big)\big)^{-m}\bar\phi_\xi(\si)d\si &\leq&  t^m \int_0^{\al} h(\si)^{-m}\bar\phi_\xi(\si)d\si \\&&+\, \left(\frac{h(\al)}{t\al}\right)^{-m} \int_{\al}^{t\al}\si^{-m}\bar\phi_\xi(\si)d\si.
\end{eqnarray*}
Since $\bar\phi_{\xi}(\si)\leq \int_M \vp d\mu$ and $\int_{\al}^{t\al}\si^{-m}d\si\leq \frac{\al^{1-m}}{m-1}$, we obtain
\begin{eqnarray}\label{t-al}
\int_0^{t\al} \big(h\big(\frac{\si}{t}\big)\big)^{-m}\bar\phi_\xi(\si)d\si &\leq& t^m \int_0^{\al} h(\si)^{-m}\bar\phi_\xi(\si)d\si \\&& + \, t^m\al\, \frac{h(\al)^{-m}}{m-1}\int_M \vp\, d\mu.\nonumber
\end{eqnarray}
It follows from (\ref{change}) and (\ref{t-al}) the following inequality.
\begin{eqnarray}\label{integral}
\frac{2}{\ka}\int_0^{\al}h(R)^{-m}\bar\psi_\xi(R)dR \nonumber &<& \frac{h(\al)^{-m}}{m-1}\int_M \vp d\mu + \frac{1}{\al} \int_0^{\al}h(\si)^{-m}\bar\phi_\xi(\si)d\si \nonumber\\ \nonumber \\ &\leq& \frac{h(\al)^{-m}}{m-1}\int_M \vp d\mu + \sup_{\si\in(0,\al)}\left(h(\si)^{-m}\bar\phi_\xi(\si)\right).
\end{eqnarray}
Using (\ref{sup}) and (\ref{integral}) we obtain
\begin{eqnarray*}
\frac{2}{\ka}\sup_{\si\in(0,\al)}\left(h(\si)^{-m}\bar\phi_\xi(\si)\right) &<& \frac{2}{\ka}\left(h(\al)^{-m}\bar\phi_\xi(\al)\right) + \frac{h(\al)^{-m}}{m-1}\int_M \vp d\mu \nonumber \\&&+ \sup_{\si\in(0,\al)}\left(h(\si)^{-m}\bar\phi_\xi(\si)\right),
\end{eqnarray*}
hence we obtain
\begin{equation}\label{final}
(\frac{2}{\ka}-1)\sup_{\si\in(0,\al)}\left(h(\si)^{-m}\bar\phi_\xi(\si)\right)< \frac{2}{\ka}\left(h(\al)^{-m}\bar\phi_\xi(\al)\right) + \frac{h(\al)^{-m}}{m-1}\int_M \vp d\mu.
\end{equation}
We recall that $h(0)=0$, $h'(0)=1$ and $h(\al)=J=\left(\frac{\om_m^{-1}e^{f^*}}{1-\ka}\int_M \vp e^{-f}dM\right)^{\frac{1}{m}}$. Thus we obtain
\begin{eqnarray*}\left\{
\begin{array}{l}
h(\al)^{-m}\bar\phi_\xi(\al) \leq h(\al)^{-m} \int_M \vp d\mu = J^{-m} (1-\ka)\om_m J^m e^{-f^*}=(1-\ka)\om_m e^{-f^*}; \\
\displaystyle\sup_{\si\in(0,\al)}\left(h(\si)^{-m}\bar\phi_\xi(\si)\right)\geq \displaystyle\limsup_{\si\to 0} \left(h(\si)^{-m}\bar\phi_\xi(\si)\right) = \om_m \left(\vp(\xi)e^{-f(\xi)}\right)\geq \om_m e^{-f^*}.
\end{array}\right.
\end{eqnarray*}
Thus, using (\ref{final}) we obtain 
\begin{equation*}
\left(\frac{2}{\ka}-1\right)\om_m< \frac{2(1-\ka)}{\ka}\om_m + \frac{1-\ka}{m-1}\om_m,
\end{equation*}
that is, $1<\frac{1-\ka}{m-1}\leq 1$, which is a contradiction. Lemma \ref{l4.2} is proved.
\end{proof}

\section{Proof of Theorem \ref{desig-Sobolev}}
Consider the set $A=\left\{\xi\in M \bigm| \vp(\xi)\geq 1\right\}.$
Take $t>2$ so that $t\al \le R_0=\min\{\inj_{\vp},r_0\}$ and set $\be\in [\frac{2}{t},1)$. Consider the sequence $R_j=\be^j\al$, with $j=0,1,\ldots$, and define the collection of subsets
\begin{equation*}
A_j=\left\{\xi\in A \bigm| \bar\phi_\xi(tR)\leq \frac{2\al}{\ka} t^{m-1}\bar\psi_\xi(R), \mbox{ for some } R\in [\be R_j,R_j)\right\}.
\end{equation*}
By Lemma \ref{l4.2}, $A=\sqcup_{j=0}^\infty A_j$. Consider the sequence of subsets $F_k\subset A$, with $k=0,1,\ldots$, defined inductively as follows: (I): $F_0=\emptyset$; (II): Assume that $F_0, \ldots, F_{k-1}$ is defined, with $k\geq 1$. For each $\ell>0$, let $S_{\ell}(\xi)=M\cap \m B_\ell(\xi)$. Consider 
\begin{equation*}
\label{D_k} D_k=\bar A_k-\cup_{j=1}^{k-1} \cup_{\xi\in F_j} S_{t\be R_j}(\xi).
\end{equation*}   
\begin{claim}\label{subset} There exists a finite subset $F_k\subset D_k$ satisfying: 
\begin{enumerate}[(i)]
\item\label{cover} $F_k\subset D_k \subset \cup_{\xi\in F_k} S_{t\be R_k}(\xi)$;
\item\label{disj} $\m B_{R_k}(\xi)\cap \m B_{R_k}(\xi')=\emptyset$, for all $\xi\neq \xi'\in F_k$.
\end{enumerate}
\end{claim}
\begin{proof} Note that $D_k$ is compact, since $A$ is compact and $D_k$ is closed. Thus, there exists a finite subset $\m C\subset D_k$ satisfying $D_k\subset \cup_{\xi\in\m C} S_{t\be R_k}(\xi)$. Take $\xi_1\in \m C$. If $D_k\subset S_{t\be R_k}({\xi_1})$, we define $F_k=\{\xi_1\}$. Otherwise, take $\xi_2\in D_k-S_{t\be R_k}({\xi_1})$. Note that $\m B_{R_k}({\xi_1})\cap \m B_{R_k}({\xi_2})=\emptyset$, since $t\be R_k\geq 2R_k$. If $D_k\subset S_{t\be R_k}({\xi_1})\cup S_{t\be R_k}({\xi_2})$ then we define $F_k=\{\xi_1,\xi_2\}$. Using that $\m C$ is a finite set, following this steps we will obtain a finite subset $F_k$ satisfying \ref{cover} and \ref{disj}. Claim \ref{subset} is proved and the collection $F_k$, with $k\geq 0$, is defined.\end{proof}

\begin{claim}\label{collection} The collection of subsets $F_k\subset A$, with $k=0,1,\ldots$, satisfies:
\begin{enumerate}[(i)]
\item\label{finite} $F_k$ is finite and $F_k\subset D_k$;
\item\label{sum} $A\subset \cup_{j=1}^\infty \cup_{\xi\in F_j} S_{t\be R_k}(\xi)$;
\item\label{pairdisjoint} the colection  $\m B_{R_k}(\xi)$, with $\xi\in F_k$ and $k\geq 1$, are pairwise disjoint.
\end{enumerate}
\end{claim}
\begin{proof}  Item \ref{finite} it follows trivially from Claim \ref{subset}. Item \ref{sum} follows from the following facts: $D_k=\bar A_k-\cup_{j=1}^{k-1} \cup_{\xi\in F_j} S_{t\be R_j}(\xi)$, $D_k\subset \cup_{\xi\in F_k} S_{t\be R_k}(\xi)$ and $A\subset \cup_{k=0}^\infty \bar A_k$. To prove Item \ref{pairdisjoint}, take  $\xi\in F_j$ and $\xi'\in F_{k}$ with $j\leq k$. If $j=k$ then $\m B_{R_k}(\xi)\cap \m B_{R_k}(\xi')=\emptyset$, by Item \ref{disj} of Claim \ref{subset}. If $j\leq k-1$ then since $F_k\subset \bar A_k-\cup_{j=1}^{k-1} \cup_{\xi\in F_j} S_{t\be R_j}(\xi)$, we obtain that $\xi'\not\in S_\xi(t\be R_j)$. This implies that  $\m B_{R_j}(\xi)\cap \m B_{R_k}(\xi')=\emptyset$, since $t\be\geq 2$ and $0<R_k\leq R_j$. Claim \ref{collection} is proved. \end{proof}

For each $\xi\in F_k$, it holds that $\bar\phi_\xi(tR)\le\frac{2\al}{\ka} t^{m-1}\bar\psi_\xi(R)$, for some $R\in (\be R_k,R_k]$. This implies that
\begin{equation*}
\bar\phi_\xi(t\be R_k)\leq \bar\phi_\xi(tR)\leq \frac{2\al}{\ka}t^{m-1}\bar\psi_\xi(R) \leq \frac{2\al}{\ka}t^{m-1}\bar\psi_\xi(R_k).
\end{equation*}
Thus, since $\vp(\xi)\geq 1$, for all $\xi\in A$, it follows by Claim \ref{collection} the following. 
\begin{eqnarray}\label{volA}
\vol_f(A)&\leq& \int_A \vp d\mu \leq \sum_{k=1}^\infty\sum_{\xi\in F_k} \bar\phi_\xi(t\be R_k) \leq 
\sum_{k=1}^\infty\sum_{\xi\in F_k} \frac{2\al}{\ka}t^{m-1}\bar\psi_\xi(R_k)\nonumber \\
&=& \frac{2\al}{\ka}t^{m-1}\int_{\cup_{k=1}^\infty \cup_{\xi\in F_k}\m S_{R_k}(\xi)} |\na \vp+\vp(H_f-\bar\na f)| d\mu \nonumber
\\&\leq& \frac{2\al}{\ka}t^{m-1}\int_M |\na \vp+\vp(H_f-\bar\na f)|d\mu.
\end{eqnarray}

Now, for each $s>0$, we define the set
$A^s=\left\{\xi\in M \bigm| \vp(\xi)\geq s\right\}$,
and let $\bar J=\bar J(\ka,\vp)$ be given by 
$$\bar J=\left(\frac{\om_m^{-1}e^{f^*}}{1-\ka}\vol_f\big(\spp(\vp)\big)\right)^{\frac{1}{m}}.$$ 
Assume that $0<\bar J<s_0$, for some  $\ka\in (0,1)$ and let $\bar \al\in (0,r_0)$ be given by $h(\bar \al)=\bar J$. Assume further that $t\bar \al<R_0$, for some $t>2$.

Fix $\ep>0$ and let $\delta=\delta(\cdot\,,\ep):\real\to [0,1]$ be a non-decreasing $C^1$ function  satisfying: 
\begin{equation}\label{mu}\left\{
\begin{array}{l}
0<\delta(t)<1, \mbox{ for all } t\in (-\ep,0);\\ 
\delta(t)=0, \mbox{ for all } t \in (-\infty,-\ep];\\ 
\delta(t)=1, \mbox{ for all } t\in [0,\infty).    
\end{array}\right.
\end{equation}
For all $s>\ep$ we consider the function $\eta=\eta(\cdot\,,\ep,s):M\to \real$ given by 
\begin{equation*} 
\eta(\xi)=\delta\left(\vp(\xi)-s\right).
\end{equation*}
It is easy to see that
\begin{claim}\label{items-f} The following statements hold:
\begin{enumerate}[(i)]
\item\label{C1} $\eta \in C^1(M)$;
\item\label{fmenor1} $0\leq \eta(\xi)\leq 1$, for all $\xi\in M$; 
\item\label{sppf} $\spp\eta\subset \spp\vp$;
\item\label{As} $\eta(\xi)=1$ if, and only if, $\vp(\xi)\geq s$.
\end{enumerate}
\end{claim}

In particular, if $\spp\eta\neq \emptyset$ then  
$0<J(\ka,\eta)\leq \bar J(\ka,\vp)<r_0$,
hence Lemmas \ref{l4.1} and \ref{l4.2} applies (with $J=J(\ka,\eta)$ and $\al=\al(\ka,\eta)$). Thus, by  (\ref{volA}) and Claim \ref{items-f} , we obtain the following. 
\begin{equation}\label{As-f}
\vol(A^s)=\vol\left(\{\xi \bigm| \eta(\xi)=1\}\right) \leq \frac{2\al}{\ka}t^{m-1}\int_M |\na \vp+\vp(H_f-\bar\na f)|d\mu. 
\end{equation}

We recall that the function $h$ satisfies $h(0)=0$, $h'(0)=1$ and $h:[0,r_0)\to [0,s_0)$ is increasing and concave. Thus the inverse function $h^{-1}:[0,s_0)\to [0,r_0)$ is increasing, convex and satisfies $h^{-1}(0)=0$ and $\big(h^{-1}\big)'(0)=1$ hence,  $h^{-1}(\tau)\leq s_0 \tau$, for all $\tau\in (0,s_0)$, which implies  
\begin{equation}\label{convex}
\al=h^{-1}(J)\leq \frac{r_0}{s_0} J = C_1\left(\int_M \eta d\mu\right)^{\frac{1}{m}},
\end{equation}
where $C_1=\frac{r_0}{s_0}\left(\frac{\om_m^{-1}}{1-\ka}\right)^\frac{1}{m} e^{\frac{f^*}{m}}$. 

Note that  $s^\frac{m}{m-1}\delta(\vp -s)\leq \left(\vp +\ep\right)^{\frac{m}{m-1}}$, 
for all $s>\ep$. Thus, by (\ref{As-f}) and
(\ref{convex}), we obtain
\begin{eqnarray}\label{integrar}
s^{\frac{1}{m-1}}\vol(A^s)&\leq&\frac{2 C_1}{\ka}t^{m-1}\left(s^{\frac{m}{m-1}}
\int_M \eta d\mu\right)^{\frac{1}{m}}\int_M |\na\eta+\eta(H_f-\bar\na f)|d\mu\nonumber\\
&=&C_2 \left(\int_{M} s^{\frac{m}{m-1}}\delta\big(\vp-s\big)d\mu\right)^{\frac{1}{m}}
\int_M |\na\eta+\eta(H_f-\bar\na f)|d\mu\nonumber\\ 
&\leq&C_2\left(\int_{M} \left(\vp+\ep\right)^{\frac{m}{m-1}}d\mu\right)^{\frac{1}{m}}\int_M |\na\eta+\eta(H_f-\bar\na f)|d\mu,
\end{eqnarray}
for all $s>\ep$, where $C_2=\frac{2 C_1}{\ka}t^{m-1}$. Furthermore,
\begin{eqnarray}\label{vol(A_s)cot}
\int_0^\infty s^\frac{1}{m-1}\vol_f(A_s) ds 
&=& \int_0^\infty \int_{\left\{\xi\in M \bigm| \vp(\xi)\geq s \right\}} s^{\frac{1}{m-1}} d\mu ds\\ 
&=&\int_{\left\{(\xi,s)\in M\times \real \bigm| 0<s\leq \vp(\xi)\right\}}s^\frac{1}{m-1}d\mu ds\nonumber\\
&=&\int_M\int_0^{\vp(\xi)} s^{\frac{1}{m-1}} ds d\mu\nonumber\\ 
&=& \frac{m-1}{m}\int_M \vp^{\frac{m}{m-1}}d\mu.\nonumber
\end{eqnarray}
Using (\ref{integrar}) and (\ref{vol(A_s)cot}), we obtain that
\begin{eqnarray}\label{quasesobolev}
\int_M \vp^{\frac{m}{m-1}}d\mu 
&\leq& \frac{m C_2}{m-1}\lim_{{\ep\to 0}}\,\left(\Big(\int_{M} \Big(\vp+\ep\Big)^{\frac{m}{m-1}}d\mu\Big)^{\frac{1}{m}}\times\right.\nonumber
\\&&\left.\times\int_\ep^\infty\int_M |\na\eta+\eta(H_f-\bar\na f)|d\mu ds\right).
\end{eqnarray}
Since $0\leq \delta(t)\leq 1$,
for all $t$, and $\delta(t-s)=0$, for all $s\geq t+\ep$,  we obtain  
\begin{eqnarray}\label{mu(h-s)H}
\int_\ep^\infty \int_M  \eta |H_f-\bar\na f| \, d\mu ds
&=& \int_M \int_\ep^\infty  \delta\left(\vp -s\right)|H_f-\bar\na f|\,dsd\mu\nonumber \\ 
&=& \int_M\int_\ep^{\vp+\ep} \delta\left(\vp -s\right)|H_f-\bar\na f| \,ds d\mu\nonumber \\ 
&\leq& \int_M \vp|H_f-\bar\na f| \, d\mu.
\end{eqnarray}
Furthermore, since $|\na\eta|
=\de'\left(\vp-s\right)|\na\vp| = -\frac{d}{ds} \de(\vp-s)|\na\vp|$, we obtain from the fundamental theorem of calculus the following.
\begin{eqnarray}\label{TFC}
\int_\ep^\infty\int_M |\na\eta|\, d\mu ds &=& \int_M\int_\ep^{\vp(\xi)+\ep} |\na\eta|\, dsd\mu \nonumber\\ \nonumber \\ &\leq& \int_M\int_\ep^{\vp(\xi)+\ep}\delta'\left(\vp(\xi)-s\right)\left|\na \vp\right|\,dsd\mu \nonumber\\
&=& \int_M \de(\vp(\xi)-\ep)|\na\vp|d\mu \nonumber\\
&\le& \int_M |\na\vp|d\mu.
\end{eqnarray}
Therefore, we obtain 
\begin{equation}\label{sob-des-1}
\left(\int_M \vp^{\frac{m}{m-1}}d\mu \right)^{\frac{m-1}{m}} \le \frac{mC_2}{m-1} \int_M \left(|\na\vp|+\vp(|H_f-\bar\na f|)\right)d\mu.
\end{equation}

To finish the proof of Theorem \ref{desig-Sobolev}, we  apply (\ref{sob-des-1}) to the function $\vp^\ga$, where $\ga>1$ is a constant to be defined. By H\"older inequality, we obtain
\begin{eqnarray}\label{holder-ineq}
\left(\int_M \vp^{\frac{\ga m}{m-1}}d\mu \right)^{\frac{m-1}{m}} &\le& C_3 \int_M \vp^{\ga-1}\left(|\na\vp|+\vp(|H_f-\bar\na f|)\right)d\mu \nonumber\\ &\le& C_3 \left(\int_M \vp^{q(\ga-1)}\right)^{\frac{1}{q}}\left(\int_M\left(|\na\vp|+\vp(|H_f-\bar\na f|)\right)^p d\mu\right)^{\frac{1}{p}},
\end{eqnarray}
where $C_3=\frac{m C_2}{m-1}$ and $q=\frac{p}{p-1}$. Take $1<p<m$ and let $\ga=\frac{p(m-1)}{m-p}$. We have that $\frac{\ga m}{m-1} = q(\ga-1)=\frac{mp}{m-1}$ and $\frac{m-1}{m}-\frac{1}{q}=\frac{m-p}{mp}$. Thus, by (\ref{holder-ineq}), we obtain
\begin{equation}
\left(\int_M \vp^{\frac{mp}{m-p}}d\mu\right)^{\frac{m-p}{m}} \le C_3 \int_M\left(|\na\vp|+\vp(|H_f-\bar\na f|)\right)^p d\mu.
\end{equation}
We obtain the constant $S$ as in (\ref{sobolev-constant}) by taking $t\to 2$ in $$\lim_{{t\to 2}\atop{t>2}} C_3= \frac{2mr_0}{ks_0(m-1)}2^{m-1}\left( \frac{\om_m^{-1}}{1-\ka}\right)^\frac{1}{m}e^{\frac{f^*}{m}}=S e^{\frac{f^*}{m}}.$$ Theorem \ref{desig-Sobolev} is proved.

\section{Proof of Theorem \ref{des_isoperimetrica}}

\begin{proof} Consider the neighborhood $V=\{x\mid d_M(x,\p M)<\ep\}$. Take $A>1$ and let  $\vp=\vp(\cdot\,,\ep):M\to \real$ be a nonnegative $C^1$ function satisfying:
\begin{enumerate}[(i)]
\item $\vp(x)=1, \mbox{ if } d_M(x,\p M)\geq \ep$;
\item $0<\vp(x)<1 \mbox{ and } |\na \vp|\leq A\ep^{-1}, \mbox{ if } 0<d_M(x,\p M)<\ep$;
\item $\vp|_{\p M}=0$.
\end{enumerate}
By Theorem \ref{desig-Sobolev} we obtain 
\begin{equation*}
\dfrac{1}{S}\left(\int_{\{ \xi \mid \rho(\xi)\geq \ep\}} \,d\mu\right)^{\frac{m-1}{m}} \leq  \int_M |\na \vp| d\mu + \int_M |H_f-\bar\na f|d\mu,
\end{equation*}
provided that condition (\ref{cond-isoper-inq}) holds.
Using that  $|\na \rho|=1$, everywhere in $V$, it follows from the coarea formula that 
\begin{equation*}
\int_M |\na \vp|d\mu = \int_M |\na \vp| e^{-f} dM \leq \frac{A}{\ep} \int_0^\ep\int_{\{\xi \mid \rho(\xi)=\tau\}} e^{-f} d\m H^{m-1}
\end{equation*}
Since $\p M=\{\xi\mid d_M(\xi,\p M)=0\}$, by taking $\ep\to 0$, we obtain that 
\begin{equation*}
\int_M |\na \vp| d\mu \leq A \int_{\p M} \, e^{-f} d\m H^{m-1} = \vol_f(\p M).
\end{equation*}
Therefore, it holds
\begin{equation*}
\dfrac{1}{S}\vol_f(M)^{\frac{m-1}{m}} \leq \vol_f(\p M) + \int_M |H_f-\bar\na f|d\mu
\end{equation*}
Theorem \ref{des_isoperimetrica} is proved.
\end{proof}

\section{Proof of Theorem \ref{cmv-proc}.}

Let $\m K\in C^0(\real)$ be a nonnegative even function such that the radial curvatures of $\bar M$ satisfy $(\bar K_\rad)_{\xi}\leq \m K(r_\xi)$, for all $\xi\in M$. Let $h:(0,r_0)\to (0,s_0)$ be an increasing solution of (\ref{cauchy}) with $(0,s_0)=h(0,r_0)$.
Assume by contradiction that an end $E$ of $M$ has finite $f$-volume. Let $B=B_{\la_0}(\xi)$ be a geodesic ball of $M$ of radius $\la_0$ and center $\xi$. Take $\la_0$ sufficiently large so that $\p E\subset B$ and $\vol_f(E-B)<\La$, where $0<\La<1$ is a small constant satisfying 
\begin{equation}\label{isop-guarantee}
\begin{array}{l}
\bar J_\La=\left(\dfrac{\om_m^{-1}e^{f*}}{1-\ka}\La\right)^{\frac{1}{m}}\le s_0;\\
h^{-1}(\bar J_\La)\le 2 \inj_M,
\end{array}
\end{equation}
for some $\ka\in (0,1)$. Moreover, we take $\la_0$ sufficiently large satisfying further  
\begin{equation}\label{lp-cond}
\begin{array}{l}
\|H_f-\bar\na f\|_{L^p_\mu(E-B)}<\m C, \mbox{ if } m\le p <\infty;\\
\|H_f-\bar\na f\|_{L^\infty(E)} \vol_f(E-B)^{\frac{1}{m}} < \m C, \mbox{ if } p=\infty,
\end{array}
\end{equation}
where  $2\m C=(S e^{\frac{f^*}{m}})^{-1}$.

Now take $\la_1>\la_0$ sufficiently large so that $d_M(\p E,x)>2\la_0$, for all $x\in E-B_{\la_1}$. For all $q\in E-B_{2\la_1}$ we obtain that the ball $B_{\la_1}(q)\subset E-B$. In particular, by (\ref{isop-guarantee}), Theorem \ref{des_isoperimetrica} applies for $B_r(q)$, for all $0<r<\la_1$.  By H\"older inequality, we obtain that
\begin{equation}\label{cond-holder}
\begin{array}{l}
\int_{B_r(q)} |H_f-\bar\na f| d\mu \le \|H_f - \bar\na f\|_{L^p_\mu(E-B)}\vol_f(B_r(q))^{\frac{p-1}{p}}, \mbox{ if } m\le p<\infty;\\
\int_{B_r(q)} |H_f - \bar\na f| d\mu \le \|H_f - \bar\na f\|_{L^\infty_\mu(E)}\vol_f(E-B)^{\frac{1}{m}}\vol_f(B_r(q))^{\frac{m-1}{m}}.
\end{array}
\end{equation}
Furthermore, if $m\le p<\infty$ then, since $\vol_f(B_r(q))\le \vol_f(E-B)<1$ and $\frac{(p-1)}{p}\ge \frac{(m-1)}{m}$,  it holds that   $\vol_f(B_r(q))^{\frac{p-1}{p}}\ge \vol_f(B_r(q))^{\frac{m-1}{m}}$. Thus, by Theorem \ref{des_isoperimetrica}, and using  (\ref{lp-cond}) and (\ref{cond-holder}), we obtain 
$\m C\,\vol_f(B_r(q))^{\frac{m-1}{m}} \le \vol_f(\p B_r(q)).$
Thus, by using the coarea formula, 
$$\frac{d}{dr}\vol_f(B_r(q))^{\frac{1}{m}}=m^{-1}\vol_f(B_r(q))^{1-\frac{1}{m}}\vol_f(\p B_r(q))\ge \m C,$$
for all $0<r<\la_1$, hence $\vol_f(B_{\la_1}(q))\ge \m C \la_1$. 

Since $M$ is complete and $E$ is an unbounded connected component of $M$ we can take $q_k\in E - (B_{2k\la_1}-B_{(2k-1)\la_1})$, for all $k=1,2,\ldots$. Note also that $B_{\la_1}(q_k)\subset E-B_{2\la_1}$ and $B_{\la_1}(q_k)\cap B_{\la_1}(q_l)=\emptyset$, if $k\neq l$. Thus we obtain
$\vol_f(E)\ge \sum_{k=1}^\infty \vol_f(B_{\la_1}(q_k)) \ge \sum_{k=1}^\infty \m C \la_1=\infty$, which is a contradiction. Theorem \ref{cmv-proc} is proved.

\end{document}